\documentclass[12pt, reqno]{amsart}
\usepackage{amsmath, amsthm, amscd, amsfonts, amssymb, mathtools, color, hyperref}
\usepackage[dvipsnames]{xcolor}

\newtheorem{theorem}{Theorem}[section]
\newtheorem{lemma}[theorem]{Lemma}

\theoremstyle{definition}
\newtheorem{definition}[theorem]{Definition}

\newtheorem{question}[theorem]{Question}
\theoremstyle{remark}
\newtheorem{remark}[theorem]{Remark}
\numberwithin{equation}{section}

\begin{document}
\setcounter{page}{1}

\title[Copositive and completely positive matrices - linear preservers]
{Linear preservers of copositive and completely positive matrices}

\author[Sachindranath]{Sachindranath Jayaraman}
\address{School of Mathematics\\ 
Indian Institute of Science Education and Research Thiruvananthapuram\\ 
Maruthamala P.O., Vithura, Thiruvananthapuram -- 695 551, Kerala, India.}

\email{sachindranathj@iisertvm.ac.in, sachindranathj@gmail.com}

\author[Vatsalkumar]{Vatsalkumar N. Mer}
\address{Department of Mathematics\\ 
Chungbuk National University\\ 
Cheongju - 28644, Korea.}

\email{vnm232657@gmail.com}



\begin{abstract}
The objective of this manuscript is to understand the structure of an invertible linear 
map on the space of real symmetric matrices $\mathcal{S}^n$ that leaves invariant the 
closed convex cones of copositive and completely positive matrices ($COP_n$ and $CP_n$). A description of an invertible linear map on $\mathcal{S}^2$ such that $L(CP_2) \subset CP_2$ 
is completely determined. 
\end{abstract}

\maketitle

\section{Introduction}\label{sec-1}

We work throughout over the field $\mathbb{R}$ of real numbers.
Let $M_{m,n}(\mathbb{R})$ denote the set of all $m \times n$ matrices over $\mathbb{R}$. 
When $m = n$, this set will be denoted by $M_{m}(\mathbb{R})$ or $M_{n}(\mathbb{R})$. 
The subspace of real symmetric matrices will be denoted by $\mathcal{S}^n$. A subset $K$ 
of finite dimensional real Hilbert space $V$ is called a convex cone if $K+K \subseteq K$ 
and $\alpha K \subseteq K$ for every $\alpha \geq 0$. The {\it conic hull} of a subset $S$ 
of $V$ is defined to be 
$$\text{cone}(S) = \Big\{\displaystyle \sum_{i=1}^{m} \alpha_i x_i: x_i \in S, \alpha_i 
\geq 0, m \in \mathbb{N}\Big\}.$$ It is obvious that $\text{cone}(S)$ is a convex cone in $V$. 
If $K$ is a subset of $V$, the dual cone of $K$ is defined as 
$K^{\ast}:= \{y \in V: \langle y , x \rangle \geq 0 \ \forall x \in K \}$. 
A convex cone $K$ is said to be proper if it is topologically closed, pointed 
($K \cap -K = \{0\}$) and has nonempty interior, denoted by $K^{\circ}$. The following 
well-known facts will be used in this manuscript. A standard reference to these is 
\cite{Ben-Israel}.

\begin{itemize}\label{basic-facts-cones}
\item The dual $K^{\ast}$ is a closed convex cone for any subset $K$ of $V$.
\item If $K$ is a convex cone, then $(K^{\ast})^{\ast} = \text{closure}(K)$. A subset 
$K$ is a closed convex cone in $V$ if and only if $(K^{\ast})^{\ast} = K$.
\item A closed convex cone $K$ in $V$ is said to be self-dual if $K = K^{\ast}$.
\item If $K_1,  K_2$ are closed convex cone in $V$, then $(K_1+K_2)^{\ast} = 
K_1^{\ast} \cap K_2^{\ast}$ 
and $\text{closure}(K_1^{\ast} + K_2^{\ast}) = K_1^{\ast} \cap K_2^{\ast}$.
\end{itemize}

Three of the well known examples of proper self-dual cones commonly used in the optimization literature are 
(1) The nonnegative orthant 
$\mathbb{R}^n_+ = \{x = (x_1, \ldots, x_n)^t : x_i \geq 0, \ i = 1, \ldots, n\}$ 
of $\mathbb{R}^n$, (2) The Lorentz or the ice-cream cone 
$\mathcal{L}^n_{+} = \{x = (x_1, \ldots, x_n)^t : x_n \geq 0, \ 
x_{n}^2 - x_{n-1}^2 - \ldots x_{1}^2 \geq 0\}$ in $\mathbb{R}^n$ and (3) The set 
$\mathcal{S}^n_{+}$ of all symmetric positive semidefinite matrices in the space 
$\mathcal{S}^n$ of real symmetric matrices. (Note that $\mathbb{R}^n$ carries the standard 
inner product, while $\mathcal{S}^n$ is equipped with the trace inner product). A cone 
$K$ is said to be polyhedral if $K = X(\mathbb{R}^m_+)$ for some $X \in M_{m,n}(\mathbb{R})$ 
and simplicial if $X$ is an invertible matrix. Our focus in this manuscript is on the 
copositive and complete positive cones, as defined below.

\begin{definition}\label{cop-cp}
An $n \times n$ real symmetric matrix $A$ is called copositive if its quadratic form is 
nonnegative on the nonnegative orthant of $\mathbb{R}^n$. The matrix $A$ is said to be 
completely positive if $A = BB^t$ for some nonnegative (not necessarily square) 
matrix $B$.
\end{definition}

\medskip

The set of copositive matrices forms a closed convex cone $COP_n$ in $\mathcal{S}^n$ 
with nonempty interior and its dual $CP_n$ consists of the closed convex cone of 
completely positive matrices; moreover, each of these cones is contained 
in the dual of the other. If $\mathcal{N}^n_+$ denotes the cone of nonnegative matrices in 
$\mathcal{S}^n$, then $\mathcal{S}^n_+ + \mathcal{N}^n_+ \subseteq COP_n$, with equality 
when $n \leq 4$. Note that in such a case, $CP_n = \mathcal{S}^n_+ \cap \mathcal{N}^n_+$. 
It can also be easily proved that $\mathcal{N}^n_+$ is a self-dual cone in $\mathcal{S}^n$. 
In fact, $\mathcal{N}^n_+$ is isomorphic to $\mathbb{R}^{n(n+1)/2}_+$. For a comprehensive treatment of these two cones, we refer to the recent monograph by Berman and Shaked-Monderer \cite{bs-m}.

\medskip

\begin{definition}\label{defn-1}
For proper cones $K_1$ and $K_2$ in $\mathbb{R}^n$ and $\mathbb{R}^m$, respectively, 
we have the following notions. $A \in M_{m,n}(\mathbb{R})$ is 
\begin{enumerate}
\item $(K_1,K_2)$-nonnegative if $A (K_1) \subseteq K_2$.
\item $(K_1,K_2)$-positive if $A (K_1 \setminus \{0\}) \subseteq K_2^{\circ}$.
\item $(K_1,K_2)$-semipositive if there exists a $x \in K_1^{\circ}$ such that 
$Ax \in K_2^{\circ}$.
\end{enumerate}
\end{definition}

It is clear that if $A$ is $(K_1,K_2)$-positive, then $A$ is $(K_1,K_2)$-semipositive. 
We denote the set of all matrices that are $(K_1,K_2)$-nonnegative by $\pi(K_1,K_2)$. 
When $K_1 = K_2 = K$, this will be denoted by $\pi(K)$. It is a well-known fact that when 
$K$ is a proper cone, $\pi(K)^{\circ}$ is precisely the 
set of all $K$-positive elements. One may refer to \cite{tam} for more information on the 
above defined notions. Let us also denote the set of all matrices that are 
$(K_1,K_2)$-semipositive by $\text{Sem}(K_1,K_2)$. When $K_1 = K_2 = K$, this will be 
denoted by $\text{Sem}(K)$. Semipositive matrices 
occur very naturally in certain optimization problems, namely, the linear 
complementarity problem. For a complete description of this class, one may refer 
to \cite{csv-1}. This work began in trying to find possible applications 
of an interesting characterization of nonnegativity obtained recently: A matrix $A$ is 
nonnegative relative to a cone $K$ if and only if for every semipositive matrix $B$ 
(relative to $K$), the matrix $A+B$ is semipositive (again, relative to $K$)  
(see the next section for the appropriate reference). A restatement of this to a linear 
map defined on a finite dimensional real Hilbert space $V$ leaving invariant a proper 
cone $K$ in $V$ holds as well. Understanding the structure of the collection of all 
linear maps with the above property has been studied for a long time and considering the characterization of nonnegativity stated above, we were naturally led to the following 
preserver problem.\\

\begin{question}\label{qn-1}
Determine the structure of an invertible linear map $L$ on $\mathcal{S}^n$ such that 
$L(K) \subset K$, where $K$ is either $COP_n$ or $CP_n$.
\end{question}

\medskip
Suppose $K$ is the complete positive cone in $\mathcal{S}^n$. The discussion 
in the previous paragraph tells us that $L \in \pi(K) \ \iff \ L+L_1 \in \text{Sem}(K)$ 
for every $L_1 \in \text{Sem}(K)$. A recent result from \cite{csv-1} says that 
$\text{Sem}(K) = \{T_1T_2^{-1}: \ \text{where} \ T_1, T_2 \in \pi(K)^{\circ}\}$. 
We thus have the following result.

\begin{theorem}\label{starting-thm}
A linear map $L$ on $\mathcal{S}^n$ leaves invariant the complete positive cone $CP_n$ 
if and only if $L = L_5L_6^{-1} - L_3L_4^{-1}$, where $L_3, L_4, L_5$ and $L_6$ are elements 
of $\pi(CP_n)^{\circ}$ with $L_5$ and $L_6$ depending on $L_3$ and $L_4$.
\end{theorem}

Since $CP_n = \mathcal{S}^n_+ \cap \mathcal{N}^n_+$ for $n \leq 4$, we have 
$\pi(\mathcal{N}^n_+)^{\circ} \cap \pi(\mathcal{S}^n_+)^{\circ} \subseteq 
\pi(CP_n)^{\circ}$ for $n \leq 4$. While $\pi(\mathcal{N}^n_+)^{\circ}$ is not hard 
to determine, $\pi(\mathcal{S}^n_+)$ is not completely understood yet. This makes it 
difficult to determine the interior of $\pi(CP_n)$ even for small values of $n$. 
However, when $n = 2$, one can identify $\mathcal{S}^2_+$ with $\mathcal{L}^3_+$, the 
Lorentz cone in $\mathbb{R}^3$. We exploit this to describe the structure of an invertible 
linear map on $\mathcal{S}^2$ such that $L(CP_2) \subset CP_2$.\\

\medskip
This manuscript is organized as follows. Section \ref{sec-1} is introductory. 
The main results are presented in Section \ref{sec-2}, which is subdivided into five 
subsections for ease of reading. Each of these subsections is self-explanatory.

\section{Main Results}\label{sec-2}

The main results are presented in this section.

\subsection{A brief detour into preservers of $COP_n$} \hspace*{\fill} \\
\label{sec-2.1}
We begin with a breief description of the question that we are interested in. Any 
linear map on $\mathcal{S}^n$ can be expressed as 
$L(X) = \displaystyle \sum_{i=1}^{n(n+1)/2} A_iXB_i$ for appropriate square matrices 
$A_i$ and $B_i$ in $M_n(\mathbb{R})$. A standard map $L$ on $M_{m,n}(\mathbb{R})$ 
is a map of the form $X \mapsto AXB$ for matrices $A$ and $B$ of appropriate sizes. Such 
a map $L$ is invertible if and only if $A$ and $B$ are invertible. A linear preserver is 
a linear map on a space of matrices that preserves a subset 
$\mathcal{A}$ or a relation $\mathcal{R}$. There is rich history on this topic within 
linear algebra/matrix theory as can be evidenced from MathSciNet. A good reference to linear preservers is \cite{pll et al}. The monographs \cite{zhang-tang-cao} and \cite{jst} collect 
several interesting resouces on preserver problems. There are two types of preserver problems. 
The first one is to determine the structure of a map $L$ defined on a matrix space such 
that $L(K) = K$ ($K$ is a subset of the matrix space). These are called {\it onto/strong} preservers. The other one is to determine the structure of $L$ such that $L(K) \subset K$, 
which are called {\it into} preservers. Strong preservers are many times tractable, especially 
when $K$ contains a basis for the underlying space. In this case, a strong preserver is an 
{\it into} preserver that is invertible with $L^{-1}$ being an {\it into} 
preserver. Many linear preservers arise through standard maps, although there are 
exceptions. For instance, {\it into} preservers of copositive matrices are not in standard 
form. It is not hard to construct a map that preserves the collection of copositive matrices that 
is not in standard form. If $A$ is a strictly copositive matrix (such a matrix is 
in the interior of the copositive cone) and $B$ is a rank one completely positive matrix, 
both of the same size, then the map $L$ defined by $X \mapsto \langle A,X \rangle B$ 
preserves completely positive matrices, but is not in standard form. 


Notice that it is enough to consider linear maps that preserve either $COP_n$ or $CP_n$, 
as these cones are duals of each other. For recent results on linear preservers of 
copositive matrices, the reader may look into the paper \cite{fjz}. As the authors point 
out in \cite{fjz}, the structure of an {\it into} linear preserver of copositivity is 
subtle and remains unsolved till date. It is reasonable to believe that any 
{\it into} preserver of copositive matrices is necessarily of the form 
$\displaystyle \sum_{i} A_iXA_i^t$ for nonnegative matrices $A_i$. However, this turns 
out to be false even in the $2 \times 2$ case, as pointed out in \cite{fjz}. Several 
interesting results on linear preservers of copositive matrices were obtained in \cite{fjz}. 
Strong preservers of copositive matrices arise only through {\it nonnegative monomial 
congruence} (see Example $2$ and Corollary $12$ of \cite{gst} as well as \cite{shitov}). 
Recall that a monomial matrix is one that has exactly one entry in each row and column. 
We record this result below.

\begin{theorem}\label{gst-Shitov}
(\cite{gst, shitov}) A linear mapping $L: \mathcal{S}^n \rightarrow \mathcal{S}^n$ 
is a strong preserver of $COP_n$ if and only if $L$ is a monomial congruence.
\end{theorem}

Recall that two proper cones $K_1$ and $K_2$ on $V$ are said to be isomorphic if there is a bijective map 
$L$ on $V$ such that $L(K_1) = K_2$. When $K_1 = K_2$, we say that such a map is an automorphism of 
the cone. Theorem \ref{gst-Shitov}says that an automorphism $L$ of $L(COP_n)$ is of the form 
$L(X) = MXM^t$ for some fixed nonnegative monomial matrix $M$. Consequently, any 
automorphism of the cone $CP_n$ is also necessarily of the same form. We shall use 
this fact later on. 

\subsection{Necessary results} \hspace*{\fill} \\
\label{sec-2.2}
We state necessary results that will be used in the sequel. We begin 
with the following lemma.

\begin{lemma} \label{(K_1,K_2)-nonnegative}
(Corollary $3.3$, \cite{fhp}) Let $K_1$ and $K_2$ be proper cones in $\mathbb{R}^n$ 
and $\mathbb{R}^m$, respectively, and $S: \mathbb{R}^n \rightarrow \mathbb{R}^m$ 
be a linear map such that $S(K_1) \subseteq K_2$. Then $S^t (K_2^*) \subseteq K_1^*$.
\end{lemma}

The next result we need is the following theorem from \cite{csv-1}.

\begin{theorem}\label{cone-nonnegative-sp} 
(Theorem 2.4, \cite{csv-1}) For proper cones $K_1, K_2$ in $\mathbb{R}^n$, let 
$S \in \pi(K_1,K_2)$ be an invertible linear map on $\mathbb{R}^n$. If a matrix $A$ is 
$K_1$-semipositive, then the matrix $B = SAS^{-1}$ is $K_2$-semipositive. Conversely, 
if the cones are self-dual and if $C$ is $K_2$-semipositive, then there exists a 
$K_1$-semipositive matrix $A$ such that $C = (S^t)^{-1} A S^t$. 
\end{theorem}

And finally, a result (stated in Section \ref{sec-1}) that we would like to use to answer 
Question \eqref{qn-1}.

\begin{theorem}\label{thm-nonneg-char}
(Theorems $2.2$ and $2.3$, \cite{csv-2}) Let $A \in M_{m,n}(\mathbb{R})$ and let 
$K_1, \ K_2$ be proper cones in $\mathbb{R}^n$ and $\mathbb{R}^m$, respectively. Then 
$A+B \in \text{Sem}(K_1,K_2)$ for every $B \in \text{Sem}(K_1,K_2)$ if and only if $A \in \pi(K_1,K_2)$.
\end{theorem}

\subsection{What about standard maps?} \hspace*{\fill} \\
\label{sec-2.3}
This section is devoted to the structure of standard maps that preserve 
copositive/completely positive matrices. Since our linear map $L$ is defined on 
$\mathcal{S}^n$, the map $L$ has a standard form if and only if it is of the form 
$X \mapsto RXR^t$ for some $R \in M_n(\mathbb{R})$. It is not hard to 
prove that a standard map $L$ on $\mathcal{S}^n$ preservers copositive matrices 
if and only if $L(X) = RXR^t$ for some nonnegative matrix $R$ 
(see Theorem $2.2$ of \cite{fjz}). The general structure of {\it into} preservers of 
$COP_2$ and $CP_2$ is taken up in the next section. In what follows below, we consider 
a well known map that preserves completely positive matrices, which actually reduces to 
the standard form. 

\medskip

Observe that both $COP_n$ and $CP_n$ contain basis for $\mathcal{S}^n$. For instance, when 
$n=2$, the collection of matrices $\{E_{11}, E_{22}, J_2\}$, where $E_{ii}$ are the matrices 
with $1$ at the ${ii}^{th}$ position and $0$ elsewhere and $J_2$ is the $2 \times 2$ matrix 
of all $1$s, is a basis for $CP_2$. Consider the Lyapunov operator $\mathcal{L}_A$ on 
$\mathcal{S}^n$ defined by $X \mapsto AX + XA^t$. The first observation is the following 
theorem. 

\begin{theorem}\label{lyapunov-1}
If $\mathcal{L}_A (CP_2) \subseteq CP_2$, then $A = \alpha I$ for some $\alpha > 0$ 
and hence, $\mathcal{L}_A$ is a constant multiple of the identity map.
\end{theorem}

\begin{proof}
Let $A = \begin{bmatrix}
	      a & b\\
	      c & d
		\end{bmatrix}$. Since $CP_2 = \mathcal{S}^2_+ \cap \mathcal{N}^2_+$, it is easy 
to show that $\mathcal{L}_A (E_{11}) \in CP_2$ and $\mathcal{L}_A (E_{22}) \in CP_2$. 
These two together imply that $a \geq 0, d \geq 0, b = c = 0$. Now, we also have 
$\mathcal{L} (J_2) \in CP_2$, which implies that $a = d$. Therefore, the matrix $A$ equals 
$\alpha I$ for some $\alpha > 0$. This proves the lemma.
\end{proof}

Let us now consider the {\it generalized Lyapunov map} $\mathcal{L}_{A,B}$ on $\mathcal{S}^n$ defined by $X \mapsto AXB + B^tXA^t, \ A, B \in M_n(\mathbb{R})$ with $B$ invertible. Our 
next result is the following.

\begin{theorem}\label{lyapunov-2}
If the map $\mathcal{L}_{A,B}$ with invertible $B$ preserves $CP_2$, then 
$\mathcal{L}_{A,B} = \alpha B^tXB$ with $B \geq 0$ and for some $\alpha > 0$. Thus, the 
map $\mathcal{L}_{A,B}$ is a standard map.
\end{theorem} 

\begin{proof}
Let us denote by $C$ the matrix $A(B^t)^{-1}$. It is then easy to verify that the map 
$\mathcal{L}_{A,B} = \mathcal{L}_{C}$, the Lyapunov map induced by the matrix $C$. 
Theorem \eqref{lyapunov-1} implies that $C= \alpha I$ for some $\alpha > 0$ and therefore 
$A = \alpha B^t$. The conclusion now follows from Theorem $2.2$ of \cite{fjz}.	
\end{proof}

\begin{remark}\label{rem-1}
It is not hard to verify that the proofs of Theorems \eqref{lyapunov-1} and 
\eqref{lyapunov-2} holds for $n \leq 4$ as $CP_n = \mathcal{S}^n_+ \cap \mathcal{N}^n_+$ 
for such values of $n$. We have thus proved the following theorem.
\end{remark}

\begin{theorem}\label{standard-combined}
Let $L$ be an invertible linear map on $\mathcal{S}^n$ such that $L(CP_n) \subseteq CP_n$. 
If $L$ is a standard map, then $L = \mathcal{L}_{A,B}$ for suitable choices of nonnegative 
matrices $A$ and $B$. Conversely, for $n \leq 4$, if $L = \mathcal{L}_{A,B}$ with $B$ 
invertible, then $L$ is a standard map.	
\end{theorem}

\begin{proof}
Suppose $L$ is an invertible standard map on $\mathcal{S}^n$ such that 
$L(CP_n) \subseteq CP_n$. Then, Theorem $2.2$ of \cite{fjz} implies that there exists a 
nonnegative invertible matrix $A$ such that $L(X) = AXA^t$ for all $X \in \mathcal{S}^n$. 
Take $A_1 = A$ and $B_1 = A^t/2$. It then follows that $L = \mathcal{L}_{A_1,B_1}$. The 
converse statement follows from Theorems \eqref{lyapunov-1} and \eqref{lyapunov-2}.
\end{proof}

Let us now ask the question if under some {\it perturbation}, the map $L$ is a standard 
map. It may be too much to ask for such a thing to hold good. However, a nice and simple 
perturbation is connected to automorphisms of a proper cone. Given a proper cone $K$ in 
a finite dimensional real Hilbert space $(V ,\langle,.\rangle)$, a linear map $L$ on $V$ 
is said to have the $\mathcal{Z}$-property relative to $K$ (written $L \in \mathcal{Z}(K)$) 
if 
$$x \in K, y \in K^{\ast}, \langle x,y\rangle = 0 \Rightarrow \langle L(x),y\rangle \leq 0.$$ 
The above notion is a generalization of the notion of $\mathcal{Z}$-matrices. Without 
getting into much of the literature on this, let us mention a recent result of 
Gowda {\it et al}. 

\begin{theorem}\label{gowda-song-kcs}
(\cite{gss}) For a proper cone $K$ in a finite dimensional real Hilbert space $V$, the 
following two statements for a linear transformation on $V$ are equivalent.
\begin{enumerate}
\item $L \in \mathcal{Z}(K) \cap \pi(K)$.
\item $L + tI \in Aut(K)$ for all $t > 0$.
\end{enumerate}
\end{theorem}

There are other statements equivalent to the ones stated above. These may be found in 
Theorem $2$ of \cite{gss}. It easily follows from the above theorem and Theorem 
\ref{gst-Shitov} that these notions - $\mathcal{Z}(K)$ and $\pi(K)$ - cannot co-exist 
in the case of $COP_n / CP_n$. For the sake of completeness, we present a proof.

\begin{theorem}\label{Z vs Nonnegativity}
For $n \leq 4$, let $L$ be a linear transformation on $\mathcal{S}^n$ such that 
$L(CP_n) \subseteq CP_n$. Then $L$ cannot have $\mathcal{Z}$-property relative to $CP_n$.
\end{theorem} 

\begin{proof}
Suppose $L \in \mathcal{Z}(CP_n) \cap \pi(CP_n)$. Then by Theorem \eqref{gowda-song-kcs}, 
$L +tI$ must be an automorphism of $CP_n$ for every $t > 0$. Thus, $(L + tI)(X) = MXM^t$ 
for some fixed nonnegative monomial matrix $M$. In particular, $L(I) = M^tM - tI \in CP_n 
= \mathcal{S}^n \cap \mathcal{N}^n_+$ for every $t > 0$, which is certainly not possible. 
This proves the result. 
\end{proof}

\subsection{General {\it into} preservers of $CP_2/COP_2$} \hspace*{\fill} \\
\label{sec-2.4}
Let us now try to answer Question \eqref{qn-1}. The general problem of determining the 
structure of {\it into} linear preservers of $COP_n$ or $CP_n$ is very hard to tackle. 
As pointed out in Section \ref{sec-1}, an answer in the $n=2$ can be 
obtained and this section is devoted to this. In the next section, we indicate a possible 
approach to determining the general structure. As pointed out in Lemma 
\ref{(K_1,K_2)-nonnegative}, it is enough to consider Question \eqref{qn-1} when $K = CP_2$. 
We also infer from Theorems \ref{starting-thm} and 
\ref{thm-nonneg-char} that the problem reduces to determining $\text{Sem}(CP_2)$. The first 
result in this connection is the following.

\begin{lemma}\label{lemma-1}
For $n \leq 4$, the following holds: $\text{Sem}(CP_n) = \text{Sem}(\mathcal{S}^n_+) \cap 
\text{Sem}(\mathcal{N}^n_+)$.
\end{lemma}

\begin{proof}
Let $L \in \text{Sem}(CP_n)$. Then, there exists $A \in (CP_n)^{\circ}$ such that 
$L(A) \in (CP_n)^{\circ}$. Since $CP_n = \mathcal{S}^n_+ \cap \mathcal{N}^n_+$, we 
see that $(CP_n)^{\circ} = (\mathcal{S}^n_+)^{\circ} \cap (\mathcal{N}^n_+)^{\circ}$. 
Thus, the matrix $A \in (\mathcal{S}^n_+)^{\circ}$ such that $L(A) \in 
(\mathcal{S}^n_+)^{\circ}$ as well as $A \in (\mathcal{N}^n_+)^{\circ}$ with 
$L(A) \in (\mathcal{N}^n_+)^{\circ}$. This implies that $\text{Sem}(CP_n) \subseteq 
\text{Sem}(\mathcal{S}^n_+) \cap \text{Sem}(\mathcal{N}^n_+)$. This proves one inclusion. 
The other inclusion can be proved similarly.
\end{proof}

\medskip

Let us now determine $\text{Sem}(\mathcal{N}^n_+)$. Notice that $\mathcal{N}^n_+$ is a 
polyhedral cone as it is isomorphic to $\mathbb{R}^{n(n+1)/2}_+$. 

\begin{lemma}\label{lemma-2}
Let $K_1$ and $K_2$ be proper cones in finite dimensional real Hilbert spaces $V_1$ 
and $V_2$ that are isomorphic through a map $T$; that is $T: V_1 \rightarrow V_2$ is 
an invertible linear map such that $T(K_1) = K_2$. Then, 
\begin{itemize}
\item $\pi(K_1) = T^{-1} \pi(K_2) T$.
\item $\text{Sem}(K_1) = T^{-1} \text{Sem}(K_2) T$.
\end{itemize}
\end{lemma}

\medskip

Thus, to determine $\text{Sem}(\mathcal{N}^n_+)$, it suffices to determine 
$\text{Sem}(\mathbb{R}^{n(n+1)/2}_+)$. Elements of the set 
$\text{Sem}(\mathbb{R}^{n(n+1)/2}_+)$ can be characterized as $YX^{-1}$ for some 
(entrywise) positive matrices $X$ and $Y$ with $X$ invertible 
(see Theorem $2.3$, \cite{csv-1}). We thus have the following:

\begin{lemma}\label{sem-1}
The set $\text{Sem}(\mathcal{N}^n_+)$ can be determined as the set of all maps on 
$\mathcal{S}^n$ of the form $T^{-1} (YX^{-1}) T$, where $X \ \text{and} \ Y$ are entrywise positive matrices in $\mathbb{R}^{n(n+1)/2}$ with $X$ invertible and $T$ is the isomorphism 
between $\mathcal{N}^n_+$ and $\mathbb{R}^{n(n+1)/2}_+$.
\end{lemma}

\begin{proof}
There is a natural isomorphism $T$ between $\mathcal{N}^n_+$ and $\mathbb{R}^{n(n+1)/2}$. 
The proof now follows from the second statement of Lemma \ref{lemma-2} and the fact 
that semipositive maps over the nonnegative orthant of $\mathbb{R}^k$ are of the form 
$YX^{-1}$ for entrywise positive matrices with $X$ invertible.
\end{proof}

\medskip

Let us now determine $\text{Sem}(\mathcal{S}^2_+)$. Although Theorem $2.3$ of 
\cite{csv-1} gives a characterization of semipositive maps over all proper cones, 
determining the structure of positive maps over $\mathcal{S}^n_+$ (and possibly 
many other proper cones) is extremely challenging. In fact, the answer is not known for 
$\mathcal{S}^n_+$. This justifies our focus on the $n=2$ case. We will use 
Theorem \ref{cone-nonnegative-sp} here. Recall that the Lorentz cone $\mathcal{L}^3_+$ is 
isomorphic to $\mathcal{S}^2_+$ via the map $(x_1,x_2,x_3)^t \mapsto \begin{bmatrix}
									  								   x_3-x_1 & x_2\\
									  								   x_2 & x_3+x_1
									  								   \end{bmatrix}$. 
We therefore have $\text{Sem}(\mathcal{S}^2_+) = T_1^{-1} S(\mathcal{L}^3_+) T_1$, where 
$T_1$ is the above isomorphism from the Lorentz cone to $\mathcal{S}^2_+$. Thus, 
the problem now reduces to determining semipositive maps over the Lorentz cone.

\medskip

Theorem \ref{cone-nonnegative-sp} says that if $A$ is $\mathcal{L}^3_+$-semipositive, 
then there exists a semipositive matrix $B$ (with respect to the usual nonnegative orthant) 
such that $A = (S^t)^{-1} B (S^t)$, where $S \in \pi(\mathbb{R}^3_+, \mathcal{L}^3_+)$ is an 
invertible map. For instance, one can take $S$ to be the map 
$(x_1,x_2,x_3)^t \mapsto (x_1,x_2,x_1+x_2+x_3)^t$. We thus have the following:

\begin{lemma}\label{sem-2}
The set $\text{Sem}(\mathcal{S}^2_+)$ can be determined as the set of all maps on 
$\mathcal{S}^2$ of the form $(S^t T_1^{-1})^{-1} (YX^{-1}) (S^t T_1^{-1})$, where 
$X \ \text{and} \ Y$ are $3 \times 3$ entrywise positive matrices with $X$ invertible, 
$T_1$ is the isomorphism between $\mathcal{S}^2_+$ and $\mathcal{L}^3_+$ and $S$ is 
the invertible map described above.
\end{lemma}

\begin{proof}
The preceding discussion says that 
$\text{Sem}(\mathcal{S}^2_+) = \{T_1 A T_1^{-1}: A \in \text{Sem}(\mathcal{L}^3_+)\}$, 
which is equal to 
$\{(T_1^{-1})^{-1} (S^t)^{-1} B S^t T_1^{-1}: B \in \text{Sem}(\mathbb{R}^3_+)\}$. The 
proof follows since $B = YX^{-1}$ for entrywise positive matrices $X$ and $Y$.
\end{proof}

\medskip

The following theorem gives the structure of an invertible linear map on $\mathcal{S}^2$ 
that leaves invariant the cone $CP_2$.

\begin{theorem}\label{structure-thm}
An invertible linear map $L$ satisfies $L(CP_2) \subset CP_2$ if and only if for every 
pair $(X,Y)$ of $3 \times 3$ entrywise positive matrices with $X$ invertible, there exists 
another pair $(X_1,Y_1)$ of $3 \times 3$ entrywise positive matrices with $X_1$ invertible 
such that $L =  T^{-1}(Y_1X_1^{-1} - YX^{-1})T$, where $T$ is the isomorphism 
between $\mathcal{N}^2_+$ and $\mathbb{R}^{3}_+$.
\end{theorem}

\begin{proof}
When $n = 2$, the cone $\mathcal{N}^2_+$ is isomorphic to $\mathbb{R}^3_+$. Let 
us assume that this isomorphism is given by the map $T$ (which is the same as the one 
that appears in Lemma \ref{sem-1}. Therefore, given any pair $(X,Y)$ of $3 \times 3$ 
entrywise positive matrices with $X$ inertible, one can generate the elements of 
$\text{Sem}(\mathcal{N}^2_+)$ and $\text{Sem}(\mathcal{S}^2_+)$. Notice that 
$\text{Sem}(\mathcal{N}^2_+) \subset \text{Sem}(\mathcal{S}^2_+)$ and therefore 
$\text{Sem}(CP_2) = \text{Sem}(\mathcal{N}^2_+)$. Then $L(CP_2) \subset CP_2$ if and 
only if for every $L_1 \in \text{Sem}(CP_2)$, the map $L + L_1 \in \text{Sem}(CP_2)$ 
(by Theorem \ref{(K_1,K_2)-nonnegative}). This is equivalent to saying that for every 
$CP_2$-semipositive map $L_1$, the map $T(L + L_1)T^{-1}$ is a $3 \times 3$ semipositive 
matrix with respect to $\mathbb{R}^3_+$. We thus conclude that 
$T(L + L_1)T^{-1} = Y_1X_1^{-1}$ for some $3 \times 3$ entrywise positive matrices $X_1$ 
and $Y_1$, with $X_1$ invertible. Combining this with Lemma \ref{sem-1}, we complete 
the proof.
\end{proof}

\medskip

A few remarks are in order.
\begin{remark}
\
\begin{itemize}
\item The proof of Theorem \ref{structure-thm} very much depends on $n=2$, as 
determining $\text{Sem}(\mathcal{S}^n_+)$ involves the knowledge of maps that leave 
$\mathcal{S}^n_+$ invariant. (More on this will be discussed in the next section).
\item The matrices $X_1$ and $Y_1$ depend on the choice of the matrices $X$ and $Y$. 
In particular, if $X$ is a $3 \times 3$ invertible positive matrix, taking  $Y = X$ we 
deduce that $L$ has the form $L = T^{-1}(Y_1X_1^{-1} - I)T$ for some pair $(X_1,Y_1)$ of 
$3 \times 3$ entrywise positive matrices with $X_1$ invertible.
\end{itemize}
\end{remark}

\medskip

Notice that the $3 \times 3$ matrix $Y_1X_1^{-1} - YX^{-1}$ cannot be a non-positive 
scalar matrix, as $L$ preserves $CP_2$. Therefore, from Theorem $5.2$ of \cite{dgjt}, 
we infer that it is similar to a semipositive matrix, say $B_{X,Y}$. Thus, 
$L = T^{-1}W^{-1}B_{X,Y}WT$ for some invertible matrix $W$. It is possible to simplify 
this expression further, as given below. The $3 \times 3$ matrix $B_{X,Y}$ is orthogonally 
similar to 
\begin{itemize}
\item a matrix $C_{X,Y}$ that is a sum of a skew-symmetric matrix and a diagogal matrix.
\item a matrix $\widetilde{D}_{X,Y}$ that is a product of an orthogonal matrix and a 
diagonal matrix.
\end{itemize}

A proof of these statements may be found in Lemma $9$ of \cite{vermeer}. With these, it 
is now possible to reduce further the expression of $L$ from Theorem \ref{structure-thm}.

\medskip

Let us end this section by noticing that a similar representation holds for a linear 
map on $\mathcal{S}^2$ such that $L(\mathcal{S}^2_+) \subset \mathcal{S}^2_+$. The only 
difference is that instead of an isomorphism $T$ between $\mathcal{N}^2_+$ and 
$\mathbb{R}^3_+$, we have to take the map $(S^t)T_1$ as given in Lemma \ref{sem-2}.

\subsection{The dual of $\pi(K)$} \hspace*{\fill} \\
\label{sec-2.5}

If $K$ is a proper cone, then so are its dual and $\pi(K)$. Therefore, $\pi(K)^{\ast}$ 
is a proper cone as well. This suggests that if one can determine $\pi(K)^{\ast}$, 
then one can determine its dual, namely, $\pi(K)$, as well. Besides, 
the dual of $\pi(K)$ involves the boundary/facial structure of the cone $K$. 
For a proper cone $K$, let $K^{\ast} \bigotimes K$ be the set 
$\{k \otimes \ell: \  k \in K^{\ast}, \ell \in K\}$, where 
$k \otimes \ell$ is the operator $y \mapsto \langle \ell,y \rangle k$. It is well known 
that for a proper cone $K$, $\pi(K)^{\ast} = \text{cone}\{K^{\ast} \bigotimes K\}$ 
(see for instance the explanation in \cite{orlitzky} and the references cited therein). 
We shall describe below a possible way to determine $\pi(K)^{\ast}$, when $K$ is 
the cone $\mathcal{CP}_2$. Since $\mathcal{CP}_2 = \mathcal{N}^2_+ \cap \mathcal{S}^2_+$, 
we see that $\pi(\mathcal{N}^2_+) \cap \pi(\mathcal{S}^2_+) \subseteq \pi(\mathcal{CP}_2)$. 
Therefore, $\pi(\mathcal{CP}_2)^{\ast} \subseteq \Big\{\pi(\mathcal{N}^2_+) \cap 
\pi(\mathcal{S}^2_+)\Big\}^{\ast}$. As pointed out in \ref{basic-facts-cones}, the right 
hand side of the above containment equals 
$\text{closure}\Big\{\pi(\mathcal{N}^2_+)^{\ast} + \pi(\mathcal{S}^2_+)^{\ast} \Big\}$. 
Let us now determine each of these two objects. 

\begin{theorem}\label{thm-2.5.1}
$\pi(\mathcal{N}^n_+)^{\ast}$ is the self-dual cone of all $n(n+1)/2 \times n(n+1)/2$ 
nonnegative matrices (with respect to the trace inner product).
\end{theorem}

\begin{proof}
Recall that $\mathcal{N}^n_+$ is isomorphic to the nonnegative orthant in 
$\mathbb{R}^{n(n+1)/2}$. One can now use the representation of $\pi(\mathcal{N}^n_+)^{\ast}$ 
to deduce that $\pi(\mathcal{N}^n_+)^{\ast}$ is the self-dual cone of all 
$n(n+1)/2 \times n(n+1)/2$ nonnegative matrices (with respect to the trace inner product).
\end{proof}

In particular, $\pi(\mathcal{N}^2_+)^{\ast}$ is the self-dual cone of all $3 \times 3$ 
nonnegative matrices (with respect to the trace inner product). Thus, \\
$$\pi(\mathcal{CP}_2)^{\ast} \subseteq \text{closure}\Big\{\pi(\mathcal{N}^2_+) + 
\pi(\mathcal{S}^2_+)^{\ast} \Big\} \\
= \text{closure}\Big\{M_3(\mathbb{R})^+ + \pi(\mathcal{S}^2_+)^{\ast} \Big\},$$ 
where $M_3(\mathbb{R})^+$ is the self-dual cone of $3 \times 3$ nonnegative matrices. \\

\medskip
Let us now determine $\pi(\mathcal{S}^n_+)^{\ast}$. 

\begin{theorem}\label{thm-2.5.2}
$\pi(\mathcal{S}^n_+)^{\ast} = \Big\{\displaystyle \sum_{i=1}^{m} \text{trace}(A_iY)B_i, \ 
A_i,B_i \in \mathcal{S}^n_+, Y \in \mathcal{S}^n, m \in \mathbb{N} \Big\}$.
\end{theorem}

\begin{proof}
Since $\mathcal{S}^n_+$ is a self-dual cone in $\mathcal{S}^n$, we see that 
$\pi(\mathcal{S}^n_+)^{\ast} = \text{cone}\{\mathcal{S}^n_+ \bigotimes \mathcal{S}^n_+ \}$. 
Any element of $\mathcal{S}^n_+ \bigotimes \mathcal{S}^n_+$ is an operator on $\mathcal{S}^n$ 
of the form $Y \mapsto \text{trace}(AY)B$, where $A, B \in \mathcal{S}^n_+$. The result 
follows from this.
\end{proof}

We thus conclude that 
$$\pi(\mathcal{CP}_2)^{\ast} \subseteq \text{closure}\Big\{M_3(\mathbb{R})^+ + 
\displaystyle \sum_{k=1}^{m} \text{trace}(A_iY)B_i, \ A_i,B_i \in \mathcal{S}^2_+, 
Y \in \mathcal{S}^2, \ m \in \mathbb{N}\Big\}.$$ 
Notice that this calculation carries over for $n \leq 4$ with appropriate modifications, 
as $\mathcal{CP}_n = \mathcal{N}^n_+ \cap \mathcal{S}^n_+$ for such $n$. \\
 
\medskip
There is, however, a different way to look at $\pi(\mathcal{S}^2_+)^{\ast}$. 
As before, we shall make use of the fact that $\mathcal{S}^2_+$ 
is isomorphic to the Lorentz cone $\mathcal{L}^3_+$ in $\mathbb{R}^3$ through the map 
$T_1$ described in the previous section. Let us also denote by $\Theta(K)$ the 
set of all strong linear preservers (automorphisms) of $K$ for an arbitrary cone $K$. 
Lemma \ref{lemma-2} says that $\pi(\mathcal{S}^2_+) = T_1^{-1} \pi(\mathcal{L}^3_+) T_1$. 
Therefore, $\pi(\mathcal{S}^2_+)^{\ast} = (T_1)^t \pi(\mathcal{L}^3_+)^{\ast} (T_1^{-1})^t$. 
The problem thus reduces to determining $\pi(\mathcal{L}^3_+)^{\ast}$. 
Once again, it is possible to write a representation of an element of $L$ 
in $\pi(\mathcal{L}^3_+)^{\ast}$ as $\displaystyle \sum_{k=1}^{m} ab^t, \ a,b \in 
\mathcal{L}^3_+, m \in \mathbb{N}$ (recall that the Lorentz cone is a self-dual cone). 
However, an alternate description can be derived. The following well known result is due 
to Loewy and Schneider (see Theorem 4.7, \cite{loewy-schneider}).

\begin{theorem}\label{Loewy-Schneider-4.7}
$\pi(\mathcal{L}^n_+) = \text{cone} (\text{closure}(\Theta(\mathcal{L}^n_+)))$.
\end{theorem}

Notice also that $\Theta(\mathcal{L}^3_+) = T_1 \Theta(\mathcal{S}^2_+) T_1^{-1}$. It is 
well known that any automorphism of the cone $\mathcal{S}^n_+$ is of the form $PXP^t$ for 
some fixed invertible matrix $P$. We therefore have the following: \\
$$\Theta(\mathcal{L}^3_+) = \Big\{(T_1P)X(P^tT_1^{-1}): X \in \mathcal{S}^2, 
P \ \text{invertible} \Big\}.$$ Thus, a typical element of $\pi(\mathcal{L}^3_+)$ is a 
linear map of the form $\displaystyle \sum_{i=1}^{k} \alpha_i (T_1P_i) X (P_i^t T_1^{-1})$ 
for some $\alpha_i \geq 0, X \in \mathcal{S}^2$ and $P_i$ (not necessarily invertible), and 
$k$ varies over the set of natural numbers (notice that we are taking the 
closure of $\Theta(\mathcal{L}^3_+)$ before taking the conic hull). One can now determine 
the dual of the cone $\pi(\mathcal{L}^3_+)$ and thereby also determine 
$\pi(\mathcal{S}^2_+)^{\ast}$. \\

\medskip
To the best of our knowledge, the representation of an element of 
$\pi(\mathcal{S}^2_+)$ described above seems new, although using it to compute the 
dual of $\pi(\mathcal{S}^2_+)$ can be difficult.

\begin{remark}\label{final remark}
Here are a few final points we wish to make.
\begin{itemize}
\item From Theorem \ref{Loewy-Schneider-4.7} and the isomorphism between the Lorentz cone 
$\mathcal{L}^3_+$ and $\mathcal{S}^2_+$, we get a description of $\pi(\mathcal{S}^2_+)$. 
\item We also have $CP_k = \mathcal{N}^k_+ \cap \mathcal{S}^k_+$ for $k = 3$ 
and $4$. However, it is not possible to go through the Lorentz cone for such values of 
$k$. 
\end{itemize} 
\end{remark}   

\medskip

\noindent
{\bf Acknowledgements:} Part of this work was presented in the minisymposium on 
{\it Copositive and Completely Positive Matrices} at the $24^{th}$ ILAS meeting at 
Galway, Ireland, in June, 2022. The authors thank the organizers of the minisymposium 
for the presentation. Thanks are also due to Professor M. S. Gowda for his encouragement 
and comments as well as for pointing out the relevant results from his paper.

\bibliographystyle{amsplain}

\end{document}